\documentclass[reqno]{amsart}

\usepackage{amssymb,amsmath,amsfonts,amsthm}
\usepackage[all]{xy}
\usepackage{enumerate}
\usepackage{mathrsfs}
\usepackage{graphicx}

\usepackage{url} 
\newtheorem{pro}{Proposition}
\newtheorem{lem}[pro]{Lemma}
 
\newtheorem{theo}[pro]{Theorem}

\theoremstyle{definition}

\newtheorem{exem}[pro]{Example}
\newtheorem{ques}[pro]{Question}

\newcommand{\Z}{\mathbb{Z}}

\newcommand{\N}{\mathbb{N}}
\newcommand{\PP}{\mathbb{P}}
\newcommand{\VI}{\mathcal{B}}
\newcommand{\VJ}{\mathcal{Q}}
\newcommand{\point}{\mathfrak{p}}
\newcommand{\RI}{\mathcal{R_I}}
\newcommand{\SI}{\mathcal{S_I}}
\newcommand{\m}{\mathfrak{m}}
\newcommand{\Y}{\mathcal{Y}}
\newcommand{\ZC}{\mathcal{Z}}

\newcommand{\sat}{{\operatorname{sat}}}
\newcommand{\reg}{{\operatorname{reg}}}
\newcommand{\indeg}{{\operatorname{indeg}}}

\newcommand{\Hom}{{\operatorname{Homgr}}}
\newcommand{\codim}{{\operatorname{codim}}}
\newcommand{\Proj}{{\operatorname{Proj}}}
\newcommand{\Spec}{{\operatorname{Spec}}}

\newcommand{\Sym}{{\operatorname{Sym}}}
\newcommand{\Rees}{{\operatorname{Rees}}}

\newcommand{\regu}{{\operatorname{reg}}}

\newcommand{\pmt}[1]{\begin{pmatrix}#1\end{pmatrix}}

\def\ff{{\bf f}} 
\begin{document}

\title[]{Bound for the number of one-dimensional fibers of a projective morphism}


\author{Quang Hoa Tran}
\address {College of Education, Hue University, 34 Le Loi St., Hue City, Vietnam \& Institut de Math\'ematiques de Jussieu, Universit\'e Pierre et Marie Curie, Boite 247, 4 place Jussieu, F-75252 Paris Cedex 05, France}
\email{quang-hoa.tran@imj-prg.fr}

\date{\today}

\maketitle

\begin{abstract} Given a birational parameterization $\phi: \PP_k^2 \dashrightarrow \PP_k^3$ of an algebraic surface $\mathscr S\subset \PP_k^3$, we bound the number of 1-dimensional fibers of the canonical projection of the graph of $\phi$ onto its image. 
\end{abstract}

\section{Introduction}
Find the implicit equation for a rational curve or surface starting from a parameterization is a central problem in geometric modeling, it is called the implicitization problem. Computing the implicit equation can be solved by computing a suitable Gro\"obner basis. However, it is known to be quite slow in practice and hence is rarely used in geometric modeling (see, e.g., \cite{Hoffmann89}). An other method for finding the implicit equation is to eliminate by computing the resultant of the polynomials (see, e.g., \cite{Buse2001, Jouanolou96}). But in many applications, the resultant vanishes identically due to the presence of base points, which are the source points where the parameterization is not well defined.  Instead, the method of moving curves and surfaces was introduced by Sederberg and Chen \cite{Tom-Chen95} for the case of parameterized curves, and then was regularly extended for the case of surfaces, e.g. \cite{BCD2003,Cox-Goldman-Zhang2000,Lai-Chen16},\ldots. More recently, by using the approximation complexes, which were defined by Herzog, Simis and Vasconcelos in \cite{HSV82,HSV1983}, a new method for computing implicitizations  as the determinant of the approximation complexes have been described by in \cite{Buse-Jouanolou03,Buse-Chardin05,Chardin2006, Buse-Chardin-Jouanolou09},\ldots. \\

Besides the computation of implicit representations of parameterized curves or surfaces, in geometric modeling it is of interest to determine the singularities of the parametrization, that is, to know whether a point is being painted twice, or not. It is important to give a computationally efficient criterion to answer this question.  The matrix representations are also useful for solving this problem (see, e.g., \cite{Botbol-Buse_Chardin14,Buse2014}).\\

The main goal of this paper is to estimate the number of the particulier points on a surface. More precisely, let $k$ a field and let $\phi$ be a parameterization from $\PP_k^2$ to $ \PP_k^3$. Let $\Gamma\subset \PP_k^2\times \PP_k^3$ be the closure of the graph of $\phi$ and $\pi_2: \PP_k^2\times \PP_k^3\longrightarrow \PP_k^3$ the projection. For every  $\point\in \PP_k^3$, the fiber of $\pi_2$ at $\point$ is the subscheme 
$$ \pi_2^{-1}(\point):=\Gamma \times_ {\PP_k^3} \Spec(k(\point))\subset \PP_{k(\point)}^2$$
where $k(\point)$ its residue field. We denote by 
 $$\Y:=\{\point\in \PP_k^3\, \mid\, \dim \pi_2^{-1}(\point)=1\}.$$
 
In this paper we will propose a method to estimate the cardinality of $\Y$.  For any $\point\in \Y$, we denote by $h_\point\in R$ the defining equation of the unmixed component of the fiber $\pi_2^{-1}(\point)$. Our main results are the following.\\

\noindent \textbf{Theorem.} \textit{Let $I=(f_0,\ldots,f_3)$ be a homogeneous ideal of $R:=k[X_1,X_2,X_3]$ generated by four homogeneous polynomials of  the same degree $d$,  not all zero. Let $I^\sat:=I\colon_R(X_1,X_2,X_3)^\infty$ be the saturation of $I$ and $\mu=\inf \{\nu \mid I_\nu^\sat\neq 0\}.$ Suppose that $\VI=\Proj(R/I)$ is locally a complete intersection of dimension zero.
\begin{enumerate}
\item[(i)] If $\mu <d,$ then 
$\sum_{\point\in \Y}\deg(h_\point)\leq \mu.$
\item[(ii)] If $\mu=d$, then
\begin{displaymath}
\sum_{\point\in \Y}\deg(h_\point)\leq \begin{cases} 
4  & \text{if}  \quad d=3,\\
\left\lfloor \frac{d}{2} \right \rfloor  d-1 &\text{if} \quad d\geq 4.
\end{cases}
\end{displaymath}
\end{enumerate}}

\section{Preliminaries}
Let $R:=k[X_1,X_2,X_3]$ be a polynomial ring over the field $k$  with its standard $\Z$-grading $\deg(X_i)=1$ for all $i=1,2,3$. Suppose we are given an integer $d\geq 1$ and four homogeneous polynomials $f_0,\ldots,f_3\in R_d,$ not all zero. We denote by $I$ the ideal of $R$ generated by these polynomials and set $\VI:=\Proj(R/I)\subseteq \PP_k^2:=\Proj(R)$. Let $B:=k[T_0,T_1,T_2,T_3]$ with its standard grading and consider the rational map
\begin{align*}
\phi:\quad &\PP_k^2  -\dashrightarrow \PP_k^3\\
& \; x \longmapsto (f_0(x):\cdots: f_3(x)).
\end{align*}

We will always assume throughout this paper that $\phi$ is generically finite onto its image, or equivalently that the closed image of $\phi$ is a surface $\mathscr{S}$ in $\PP_k^3:=\Proj(B)$, and that the base locus $\VI$ of $\phi$ is supported on a finite set of points.\\

Let $\Gamma_0\subset \PP_k^2\times \PP_k^3$ be the graph of $\phi: \PP_k^2\setminus\VI \longrightarrow \PP_k^3$ and $\Gamma$ the Zariski closure of $\Gamma_0$. We have the following diagram
$$ \xymatrix@1{\Gamma\ \ar[d]_{\pi_1}\ar[rd]^{\pi_2}\ar@{^(->}[r] & \PP_k^2\times \PP_k^3 \\  \PP^2_k \ar@{-->}[r]_\phi& \PP^3_k}	$$
where the maps $\pi_1,\pi_2$ are the canonical projections. One has
$$\mathscr{S}=\overline{\pi_2(\Gamma_0)}=\pi_2(\Gamma)$$
where the bar denotes the Zariski closure. \\

Furthemore $\Gamma$ is the irreducible subscheme of $\PP_k^2\times \PP_k^3$ defined by the Rees algebra $\RI:=\Rees_R(I)$ (see \cite[Chapter II, \S 7]{Hartshorne77}). Let $S:=R\otimes_k B= R[T_0,\ldots,T_3]$ with the standard bigraded structure by the canonical grading $\deg(X_i)=(1,0)$ and $\deg(T_j)=(0,1)$ for all $i=1,2,3$ and $j=0,\ldots,3$. The natural bigraded morphism of bigraded $k$-algebras
\begin{align*}
\alpha:\quad & S \longrightarrow  \RI=\oplus_{s\geq 0} I(d)^s=\oplus_{s\geq 0} I^s(sd)\\
&  T_i \longmapsto f_i
\end{align*}
 is onto and corresponds to the embedding $\Gamma \subset \PP_k^2\times \PP_k^3$.
 
Let $\mathfrak{P}$ be the kernel of $\alpha$. Then it is a homogeneous ideal of $S$ and the part of degree one of  $\mathfrak{P}$ in $T_i$, denoted by  $\mathfrak{P}_1=\mathfrak{P}_{(\ast, 1)},$  is the module of syzygies of the  $f_i$
$$a_0T_0 +\cdots+a_3T_3 \in \mathfrak{P}_1\Longleftrightarrow a_0f_0+\cdots+a_3f_3=0.$$

Set $\SI:=\Sym_R(I)$ for the symmetric algebra of $I$. The  natural bigraded epimorphisms
\begin{align*}
S\longrightarrow S/(\mathfrak{P}_1)\simeq \SI \text{\quad and\quad} \SI\simeq S/(\mathfrak{P}_1)\longrightarrow S/\mathfrak{P}\simeq \RI
\end{align*}
correspond to the embeddings of schemes $ \Gamma \subseteq V\subset \PP_k^2\times \PP_k^3$
where $V$ is the projective scheme defined by $\SI$.\\

 Let $K$ be defined by the following exact sequence 
  $$0\longrightarrow K\longrightarrow \SI\longrightarrow \RI\longrightarrow 0.$$
  Notice that the module $K$ is supported in $\VI$ because $I$ is locally trivial outside $\VI$. \\

 Let us now denote by  $I^{\sat}$ the ideal 
 $$I:_R \m^\infty=\bigcup_{n\in \N}I:_R \m^n$$
  where $\m:=(X_1,X_2,X_3)$ is the irrelevant maximal ideal of $R$.\\
 
 Finally, we recall that for a nonzero $\Z$-graded $R$-module of finite type $M$, its initial degree is defined by
 $$\indeg(M):=\inf \{ \nu  \mid  M_\nu\neq 0\},$$
 with the convention $\indeg(M)=+\infty$ when $M=0$.
\section{Fibers of the canonical projection of the graph of $\phi$ onto $\mathscr{S}$} 
 From now on we assume that the ideal $I$ is locally a complete intersection outside $V(\m)$. Under this hypothesis, the symmetric algebra $\SI$  is projectively isomorphic to the Rees algebra $\RI$ (see for instance \cite[Theorem 4]{Buse-Chardin05}). In other words, the graph $\Gamma\subset \PP_k^2\times \PP_k^3$ of the parameterization $\phi$ is equal to  the subscheme $V$ of $\PP_k^2\times \PP_k^3$ defined by $\SI$. For every  $\point\in \PP_k^3,$ we will denote by $k(\point)$ its residue field $B_\point/\point B_\point$. The fiber of $\pi_2$ at $\point$ is the subscheme 
$$ \pi_2^{-1}(\point)=\text{Proj}(\RI\otimes_B k(\point))=\text{Proj}(\SI\otimes_B k(\point))\subset \PP_{k(\point)}^2.$$

In \cite{Botbol-Buse_Chardin14}, it is shown that $\pi_2$ only admits fibers of dimension 0 or 1.  Furthermore, it is easily seen that the number of 1-dimensional fibers is finite. Our purpose is to bound the cardinality of the set of points in $\PP_k^3$ with a 1-dimensional fiber, that is, the cardinality of the set 
$$\Y=\{\point\in \PP_k^3\, \mid\, \dim \pi_2^{-1}(\point)=1\}.$$

For any $\point\in \Y$, we will denote by $h_\point\in R$ the defining equation of the unmixed component of the fiber $\pi_2^{-1}(\point)$. We begin with an important remark in this direction.
\begin{pro}\label{Proposition4.1}
If there exists an integer $s$ such that $\nu=\indeg((I^s)^\sat)<sd$, then
$$\sum_{\point\in \Y}\deg(h_\point)\leq \nu<sd.$$
\end{pro}
\begin{proof} 
First notice that  $\Y$ is finite. By \cite[Lemma 10]{Botbol-Buse_Chardin14}, there exists a homogeneous polynomial $f\in I$ of degree $d$ such that for any $\point\in \Y$, we have
$$I=(f)+h_\point (g_{1\point}, g_{2\point},g_{3\point})\quad \text{and}\quad I^\sat \subseteq (f,h_\point)$$
for some $ g_{1\point}, g_{2\point},g_{3\point}\in R$. \\
	
Since $(f,h_\point)$ is a complete intersection ideal, it follows from \cite[Lemma 5, Appendix 6]{Zariski1960} that $(f,h_\point)^s$ is unmixed, hence saturated. Therefore for all $\point\in \Y$, we have
$$(I^s)^\sat\subseteq ((I^\sat)^s)^\sat \subseteq ((f,h_\point)^s)^\sat=(f,h_\point)^s=(f^s,f^{s-1}h_\point,\ldots, h_\point^s).$$
	
If $F\in (I^s)^\sat$  such that $\deg(F)=\nu<sd$, then $h_\point$ is a divisor of $F$. Moreover, if $\point\neq \point'$ in $\Y$, then $\gcd(h_\point,h_{\point'})=1.$ We deduce that
$$\prod_{\point\in \Y}h_\point \mid F$$
which gives
$$\sum_{\point\in \Y}\deg(h_\point)\leq \deg(F)=\nu<sd.$$
\end{proof}
 In particulier if $\indeg (I^\sat)=\delta< d$, by Proposition~\ref{Proposition4.1}, then
$$\sum_{\point\in \Y}\deg(h_\point)\leq \delta <d.$$

From now on we want to bound the cardinality of the set $\Y$, the delicate case is when the ideal $I$ satisfies $\indeg (I^\sat)=\indeg(I)= d.$ 

\begin{exem} Consider the parameterization given by
\begin{align*}
f_0&= X_2^2X_3^4-X_2^4X_3^2,\\
f_1&=X_1^4X_3^2-X_3^6,\\
f_2&=X_1^2X_2^2X_3^2-X_1^2X_2^4,\\
f_3&=X_1^4X_2^2-X_2^2X_3^4.
\end{align*}
Using {\tt Macaulay2} \cite{Macaulay2}, we see that $I$ is a saturated ideal
 with $\VI$ locally a complete intersection of dimension zero. Since $\indeg( (I^2)^\sat)= 8< 2.6=12,$ Proposition~\ref{Proposition4.1} shows that
$$\sum_{\point\in \Y}\deg(h_\point)\leq 8.$$
Precisely, $I$ admits a minimal free resolution of the form
$$\xymatrix {0\ar[r] &R(-8)^3\ar[r]^{ M}& R(-6)^4\ar[r]&R\ar[r]&R/I\ar[r]&0}$$
where the matrix $M$ is given by
$$\pmt{-X_3^2 & 0 & X_1^2\\0 & X_2^2& 0 \\ X_1^2& 0 & -X_3^2\\ -X_2^2+X_3^2& -X_3^2 & 0}.$$
We thus get $\Y=\{\point_1,\point_2,\point_3,\point_4\}$ with
\begin{align*}
\point_1& =(0:1:0:1), \;\;\quad \qquad h_{\point_1}=X_2^2-X_3^2,\quad\quad\point_2=(1:0:1:0),\;\quad h_{\point_2}=X_1^2-X_3^2,\\
\point_3 & =(1:0:-1:0), \;\;\;\quad \quad h_{\point_3}  =X_1^2+X_3^2,\qquad \point_4=(0:1:0:0),\;\quad h_{\point_4}=X_2^2,
\end{align*}
which implies
$$\sum_{\point\in \Y}\deg(h_\point)=8.$$
\end{exem}

\section{Estimation of number of one-dimensional fibers of  the canonical projection of the graph of $\phi$ onto $\mathscr{S}$}

As we already noticed  that 
$\sum_{\point\in \Y}\deg(h_\point)< d$
if $\indeg(I^\sat)<d,$ we assume that $\indeg(I^\sat)=\indeg(I)=d.$\\

Recall that $\VI=$Proj$(R/I)$ is assumed to be locally a complete intersection  of dimension zero. Under this hypothesis, the module $K$ is supported in $\m S$, hence $H_\m^i(K)=0$ for any $ i\geq 1.$ We deduce that
$$H_\m^i(\SI)\simeq H_\m^i(\RI),\, \forall i\geq 1.$$
Here we denote the $i$-th local cohomology modules of a $R$-module $M$ with support in the irrelevant maximal ideal $\m=(X_1,X_2,X_3)$ by $H_\m^i(M).$ We set
$$a_i(M):=\sup \{\mu\ \mid \ H_\m^i(M)_\mu\neq 0  \},$$
with the convention $a_i(M):=-\infty$ when $H_\m^i(M)=0.$ Recall that the Castelnuovo-Mumford regularity of $M$ is defined by
$$\reg(M):=\max_i\{ a_i(M)+i  \}.$$

Let $K_\bullet:= K_\bullet(\ff;R)$ be the Koszul complex associated to the sequence of polynomials $\ff:=(f_0,f_1,f_2,f_3)$ in $R$. We denote by $Z_i:=Z_i(\ff;R), B_i:=B_i(\ff;R), H_i:=H_i(\ff;R)=Z_i/B_i$ the module of cycles, the module of boundaries and the homology module in degree $i$, respectively (see  \cite[\S4]{Buse-Jouanolou03} or \cite[\S3]{Buse-Chardin05} for an introduction to these complexes in this context). Since the ideal $I$ is homogeneous, these modules inherit a natural structure of graded $R$-modules. Let $\mathcal{Z}_\bullet$ be the approximation complex of cycles associated to the ideal $I=(\ff)$. By definition $\mathcal{Z}_q=Z_q[qd]\otimes_R R[T_0,\ldots,T_3](-q)$ for all $q=0,\ldots, 3$. This complex is of the form
\begin{align*}
\xymatrix{
(\mathcal{Z}_\bullet):& 0\ar[r] &\mathcal{Z}_3\ar[r] ^{v_3}&\mathcal{Z}_2\ar[r]^{v_2}& \mathcal{Z}_1\ar[r]^{v_1\qquad\qquad }&\mathcal{Z}_0=R[T_0,\ldots,T_3]\ar[r]^{\;\quad\qquad v_0}&0 }
\end{align*}
where $v_1(a_0,a_1,a_2,a_3)=a_0T_0+\cdots+a_3T_3$. It is proved in \cite[Theorem 4]{Buse-Chardin05} that  the complex  $(\mathcal{Z}_\bullet)$ is acyclic under our hypothesis. Thus ($\mathcal{Z}_\bullet)$ is a resolution of  $H_0(\mathcal{Z}_\bullet)\simeq \SI.$

\begin{theo} \label{Theorem4.2} 
Let $I=(f_0,\ldots,f_3)$ be a homogeneous ideal of $R$ such that $\indeg(I^\sat)=\indeg(I)=d.$ Suppose that $\VI=\Proj(R/I)$ is locally a complete intersection of dimension zero.  Then there exists a complex $ (C_\bullet)$ of free $B$-modules
\begin{align*}
\xymatrix@R=8pt{0\ar[r]& B(-3)^l \ar[r]\ar@{=}[d]& B(-2)^m\ar[r] \ar@{=}[d]& B(-1)^n\ar@{=}[d]\ar[r]& 0\\ & C_3& C_2& C_1&} 
\end{align*}
where $ n=\dim_k H_\m^3(Z_1)_{d-1},\; m=\dim_k H_\m^3(Z_2)_{2d-1},\; l=\dim_k H_\m^3(Z_3)_{3d-1}$ with homology
\begin{align*}
H_1(C_\bullet)\simeq \oplus_{s\geq 0} H_\m^2(I^s)_{sd-1},\quad
H_2(C_\bullet)\simeq \oplus_{s\geq 0}(I^s)^\sat_{sd-1},\quad
H_3(C_\bullet)=0.
\end{align*}
\end{theo}

\begin{proof} 
One has the following graded (degree zero) isomorphisms of $R$-modules (see \cite[Lemma 1]{Buse-Chardin05}):
\begin{equation} \label{Equation4.1}
H_\m^p(Z_q)\simeq \begin{cases}
0& \text{for}\,\, p=0,1\\ H_m^0(H_{p-q})^\ast [3-4d]& \text{for} \,\,p=2\\
Z_{p-q}^\ast [3-4d]&  \text{for} \,\,p=3 
\end{cases}
\end{equation}
where $-^\ast:=\Hom_R(-,k).$\\
	
We  consider the two spectral sequences associated to the double complex $C_{q}^p=C_\m^p(\ZC_q),$ where $C_\m^\bullet(M)$ denotes the \v Cech complex on $M$ relatively to the ideal $\m$. Since $(\ZC_\bullet)$ is acyclic, one of them abuts at step two with:
\begin{equation} \label{Equation4.4}
_\infty\!^h\textbf{E}_{q}^p= \,_2^h\textbf{E}_{q}^p= \begin{cases}  H_\m^p(\SI),\quad \text{for} \quad & q=0\\ 0 ,\qquad\qquad \text{for}  \quad & q\neq 0.\end{cases}
\end{equation}
The other one gives at step one:
$$^{v}_1 \textbf{E}_{q}^p=H_\m^p (\ZC_q)=H_\m^p (Z_q)[qd]\otimes_R R[T_0,\ldots,T_3](-q)=H_\m^p (Z_q)[qd]\otimes_k B(-q).$$
	
Notice that $H_\m^p(Z_q)=0$ for $p=0,1$; $Z_3\simeq R[-4d]$ and $Z_0=R$. Therefore the first page of the vertical spectral sequence has only two nonzero lines
\[\xymatrix@C=10pt{ 0\ar[r]& H_\m^2(Z_2)[2d]\otimes_k B(-2)\ar[r]& H_\m^2(Z_1)[d]\otimes_k B(-1)\ar[r]&0\\
H_\m^3(Z_3)[3d]\otimes_k B(-3)\ar@{->}[r]& H_\m^3(Z_2)[2d]\otimes_k B(-2)\ar@{->}[r]& H_\m^3(Z_1)[d]\otimes_k B(-1)\ar@{->}[r]& H_\m^3(Z_0)\otimes_kB.	}
\]
		
In bidegree $(-1,\ast)$, we have $ H_\m^3(Z_0)_{-1}\otimes_k B=H_\m^3(R)_{-1}\otimes_k B=0.$ Therefore, we obtain the complex $(C_\bullet)$ of free $B$-modules
\begin{align*}
\xymatrix@R=8pt{0\ar[r]& B(-3)^l \ar[r]\ar@{=}[d]& B(-2)^m\ar[r] \ar@{=}[d]& B(-1)^n\ar@{=}[d]\ar[r]& 0.\\ & C_3& C_2& C_1&} 
\end{align*}
			
Let us compute the homology of $(C_\bullet)$.  The exact sequence
\begin{equation}\label{Equation4.2}
0\longrightarrow Z_1\longrightarrow R[-d]^4\longrightarrow I\longrightarrow 0
\end{equation}
gives 
$$
H_\m^2(Z_1)\simeq H_\m^1(I)\simeq H_\m^0(R/I)\simeq I^\sat/I.
$$
 Since $\indeg(I^\sat)=\indeg(I)=d,$ we have
 \begin{equation} \label{Equation4.5}
 H_\m^2(Z_1)_{d-1}\simeq  (I^\sat/I)_{d-1}=0.
 \end{equation}
		
It follows from \eqref{Equation4.1} that
$ H_\m^2(Z_2)_{2d-1}\simeq H_\m^0(H_0)^\ast_{2-2d}\simeq \Hom_r((I^\sat/I)_{2d-2},k)$ as $H_\m^0(H_0)\simeq I^\sat/I.$
Since $\reg(I)\leq 2d-2$ (see \cite[Lemma 4]{Botbol-Buse_Chardin14}), we have $I_{2d-2}=(I^\sat)_{2d-2},$ thus 
 \begin{equation}\label{Equation4.6}
H_\m^2(Z_2)_{2d-1}=0.
 \end{equation}
		 
It follows from \eqref{Equation4.5} and \eqref{Equation4.6} that
$$   (^{v}_1 \textbf{E}_{q}^p)_{(-1,\ast)}=0\quad \text{unless} \quad p=3\; \text{ and} \; q\in \{1,2,3\}.$$  
Therefore, in bidegree $(-1,\ast)$,
\begin{equation} \label{Equation4.3}
_\infty\!^v\textbf{E}_{q}^p=\,^{v}_2 \textbf{E}_{q}^p= \begin{cases}  H_q(C_\bullet), &\text{if}\;  p=3 \ \text{and} \; q=1,2,3\\ 0 , & \text{else} .\end{cases}
\end{equation}
			
By comparing the two spectral sequences \eqref{Equation4.4} and \eqref{Equation4.3}, one has:
\begin{align*}
H_1(C_\bullet)&= H_\m ^2(\SI)_{(-1,\ast)} =\oplus_{s\geq 0} H_\m^2(I^s)_{sd-1},\\
H_2(C_\bullet)& =H_\m ^1(\SI)_{(-1,\ast)}= \oplus_{s\geq 0} H_\m ^1(I^s)_{sd-1}= \oplus_{s\geq 0} H_\m^0(R/I^s)_{sd-1}=\oplus_{s\geq 0}((I^s)^\sat/I^s)_{sd-1}\\
&=\oplus_{s\geq 0}(I^s)^\sat_{sd-1}\; \text{since}\ I_{sd-1}^s=0,\\
H_3(C_\bullet)& =H_\m^0 (\SI)_{(-1,\ast)}\subseteq (\SI)_{(-1,\ast)}\subseteq R_{-1}\otimes_k B=0.
\end{align*}
\end{proof}
	
We now give an expression of $m,n,l$ in terms of the degree $d$ and $\deg(\VI)$.
\begin{lem} \label{Lemma4.3} 
Under the assumptions of Theorem~\ref{Theorem4.2},
$$ n= \deg(\VI)-\frac{1}{2}d(d+1),\; m=\deg(\VI)-d \;\text{and}\; \; l=\frac{1}{2}d(d-1).$$
As a consequence,  $n-m+l=0.$
\end{lem}
\begin{proof}
We first compute $n=\dim_k H_\m^3(Z_1)_{d-1}.$
From \eqref{Equation4.2}, we derive an exact sequence
\begin{equation*}
0\longrightarrow H_\m^2(I)\longrightarrow H_\m^3(Z_1)\longrightarrow H_\m^3(R)[-d]^4\longrightarrow 0,
\end{equation*}
hence $H_\m^3(Z_1)_{d-1}\simeq  H_\m^2(I)_{d-1}.$ We get 
$$n=\dim_k H_\m^2(I)_{d-1}=\dim_k H_\m^1(R/I)_{d-1}. $$
We now consider an exact sequence
\begin{align*}
0\longrightarrow  R/I^\sat\longrightarrow \bigoplus_{\mu\in \Z}H^0(\PP_k^2,\mathcal{O}_\VI (\mu))\longrightarrow H_\m^1(R/I)\longrightarrow 0,
\end{align*}
hence
$$\dim_k H_\m^1(R/I)_{d-1}=\dim_k H^0(\PP_k^2, \mathcal{O}_\VI(d-1)) -\dim_k (R/I^\sat)_{d-1}=\deg(\VI)-\dim_k (R/I^\sat)_{d-1}.$$
Since $\indeg(I^\sat) =d,\, \dim_k (R/I^\sat)_{d-1}=\dim_k R_{d-1}=\binom{d+1}{2}=\frac{1}{2}d(d+1)$. It follows that
$$n=\deg(\VI)-\frac{1}{2}d(d+1).$$
		
We next compute
$m=\dim_k H_\m^3(Z_2)_{2d-1}.$ It follows from \eqref{Equation4.1} that
$$ H_\m^3(Z_2)_{2d-1}\simeq (Z_1)^\ast_{2-2d}\simeq \Hom_R((Z_1)_{2d-2},k).$$
Therefore, by \eqref{Equation4.2} we have
$$m=\dim_k (Z_1)_{2d-2}=\dim_k R_{d-2}^4 -\dim_k I_{2d-2}=4\binom{d}{2}-\dim_k I_{2d-2}.$$
Since $\dim_k I_{2d-2}=\dim_k R_{2d-2}-\dim_k(R/I)_{2d-2}=\binom{2d}{2}-\deg (\VI),$ it follows that
$$m=4\binom{d}{2}-\binom{2d}{2}+\deg (\VI) = \deg(\VI)-d.    $$
		
Finally, we have
$$l= \dim_k H_\m^3(Z_3)_{3d-1}=\dim_k H_\m^3(R)_{-d-1}=\binom{d}{2}=\frac{1}{2}d(d-1).$$
\end{proof}

Now we will give a bound for the degree of the projective variety $\VI\subset \PP_k^2$ in terms of the degree $d$. Let us introduce some classical notations.\\
	
For a finitely generated graded $R$-module $M$, we denote by $HF_M, HP_M$ its Hilbert function and its Hilbert polynomial, respectively. By definition, for every $\mu, HF_M(\mu)=\dim_k(M_\mu)$ and by the Grothendieck-Serre formula \cite[Theorem 4.3.5]{Bruns-Herzog93}
$$HP_M(\mu)=HF_M(\mu)-\sum_i(-1)^iHF_{H_\m^i(M)}(\mu).$$
	
\begin{lem} \label{Lemma4.4} 
Under the assumptions of Theorem~\ref{Theorem4.2},
$$\frac{1}{2}d(d+1)\leq\deg(\VI)\leq d^2-2d+3.$$
\end{lem}
\begin{proof}  
The Hilbert function of $R/I^\sat$  is weakly increasing, hence for any $\mu$
$$\dim_k(R/I^\sat)_\mu\leq \deg(\VI).$$
Since $\indeg(I^\sat)=d, \deg(\VI)\geq \dim_k (R/I^\sat)_{d-1}= \dim_k R_{d-1}=\frac{1}{2}d(d+1).$
		
 We can assume that $ k $ is infinite. Since $\codim(I)=2$, there are homogeneous polynomials $g_1,g_2\in I$ of degree $d$ which form a regular sequence. Therefore $\mathfrak{b}:=(g_1,g_2)$ is a complete intersection ideal. Setting $J:=\mathfrak{b}:_RI, J$ is saturated and $I^\sat=\mathfrak{b}:_RJ$ since $R$ is Gorenstein. We set $\VJ=\Proj(R/J)$. By using liaison theory, we obtain $\deg(\VI)+\deg(\VJ)=\deg(\mathfrak{b}),$ hence $ \deg(\VI)=d^2-\deg(\VJ).$
		
Local duality and the liaison sequence show that
$$H_\m^1(R/J)\simeq (I^\sat/\mathfrak{b})^\ast[3-2d].$$
We thus get
$$H_\m^1(R/J)_\mu\simeq \Hom_R((I^\sat/\mathfrak{b})_{2d-3-\mu},k).$$
Since $\indeg(I^\sat)=\indeg(\mathfrak{b})=d$,  we obtain
$$H_\m^1(R/J)_\mu =\begin{cases}
0,\qquad\qquad \qquad\qquad\qquad \qquad\text{if}\; \mu > d-3\\
\Hom_R((I^\sat/\mathfrak{b})_d,k)\neq 0,\quad \;\;\text{if} \; \mu=d-3.
\end{cases}$$
Moreover, we have $H_\m^0(R/J)=J^\sat/J=0.$ We therefore conclude that
$$\reg(R/J)=d-2,$$
and hence  $\dim_k(R/J)_{d-2}=\deg(\VJ).$\\
		
On the other hand, the Hilbert function $HF_{R/J}(\mu)$ of $R/J$  strictly increases from 1 in degree 0 until $\deg(\VJ)$ in degree $d-2$.  Therefore one has $\deg(\VJ)\geq d-1.$ By the Grothendieck-Serre formula, for any $\mu$
$$HF_{R/J}(\mu)=\deg(\VJ)-\dim_k H_\m^1(R/J)_\mu.$$
Since $\dim_k H_\m^1(R/J)_{d-3}=\dim_k(I^\sat/\mathfrak{b})_d\geq 2$, we deduce $\deg(\VJ)\geq d.$\\
		
The proof will be completed by showing that $\deg(\VJ)\geq 2d-3.$ Set $c_i=\dim_k(R/J)_i$ for all $ i\geq 0$. The sequence $\underline{c}=\{c_i\, :\, i\geq 0\}$ satisfies the following properties:
\begin{enumerate}
\item[(i)] $c_0=1$ and $c_1=3$. Indeed, if $c_1<3$, then $J_1\neq 0$, there is therefore a linear form $\ell\in J_1$. We assume that $J=(\ell,w_1,w_2,\ldots,w_s).$ Since $\regu(J)=d-1$, we have $\deg(w_i)\leq d-1$ for all $i=1,\ldots s.$ Fix $i\in \{1,\ldots, s\}$ such that $\gcd(\ell, w_i)=1$. Since $\VJ \subseteq V(\ell,w_i)$, hence $\deg(\VJ)\leq d-1,$ a contradiction. 
\item[(ii)] $c_i=\deg(\VJ)$ for any $i\geq d-2$ and $c_{d-2}-c_{d-3}=\dim_k H_\m^1(R/J)_{d-3}\geq 2.$	
\end{enumerate}
		
Since  $J$ is saturated and $k$ is infinite, there is a linear form $\ell$ in $R$ such that $J:\ell\simeq J(-1)$ as graded $R$-modules. It follows that $\ell$ induces an injection $\times \ell: (R/J)[-1]\longrightarrow  R/J.$ Therefore the Hilbert function of $R/(J+\ell)$ is $ \underline{b}=\{b_i\, :\, i\geq 0\}$ with $b_0=1$ and $b_i=c_i-c_{i-1}$ for every $i\geq 1.$ Since $ b_{d-2}\geq 2$,  as a consequence of Macaulay's Theorem  $b_i\geq 2,$ for any $i=1,\ldots, d-2.$ \\
		
Indeed, suppose that  there exists an integer $i\geq 1$ such that $b_i=1,$ and let
$$\nu=\sup_{i\geq 1}\{ i\, \mid\, b_i=1\}.$$
One has $2\leq \nu\leq d-3$ and since $R/(J+\ell)$ is a standard graded $k$-algebra,  Macaulay's Theorem (see \cite[Theorem 2.2]{Stanley1978}) shows that
$$2\leq b_{\nu+1}\leq b_\nu^{<\nu>}=1^{<\nu>}=1.$$
This shows that $c_{i+1}\geq c_i +2$ for all $i=0,\ldots, d-3,$ hence
$$\deg(\VJ)\geq 2(d-2)+1=2d-3.$$
\end{proof}
	
\begin{pro} \label{Proposition4.5}
Under the assumptions of Theorem~\ref{Theorem4.2}, the following holds for any integer $s$:
$$\dim_k H_\m^2(I^s)_{sd-1}= \frac{1}{2}s(s+1)\deg(\VI)-\frac{1}{2}sd(sd+1)+\dim_k (I^s)^\sat_{sd-1}.$$
\end{pro}
\begin{proof} 
By Theorem~\ref{Theorem4.2}, for any integer $s$, one has
\begin{align*}
\dim_kH_\m^2(I^s)_{sd-1}&=n.\dim_k B_{s-1}-m.\dim_k B_{s-2}+l.\dim_k B_{s-3}+ \dim_k (I^s)^\sat_{sd-1}\\
&=n\binom{s+2}{3}-(n+l)\binom{s+1}{3}+l\binom{s}{3}+\dim_k (I^s)^\sat_{sd-1}\\
&=	n \binom{s+1}{2}-l \binom{s}{2} +\dim_k (I^s)^\sat_{sd-1}\\
&=\big(\deg(\VI)-\frac{1}{2}d(d+1)\big)\binom{s+1}{2}-\frac{1}{2}d(d-1) \binom{s}{2} +\dim_k (I^s)^\sat_{sd-1},\quad \text{by Lemma~\ref{Lemma4.3}}\\
&=\frac{1}{2}s(s+1)\deg(\VI)-\frac{1}{2}sd(sd+1)+\dim_k (I^s)^\sat_{sd-1}.
\end{align*}
\end{proof}
	
\begin{theo}\label{Theorem4.6}
Let $I$ be a homogeneous ideal of $R$ such that $\indeg(I^\sat)=\indeg(I)=3.$ If $\VI=\Proj(R/I)$ is locally a complete intersection of dimension zero, then 
$$\sum_{\point\in \Y}\deg(h_\point)\leq 4.$$
\end{theo}
\begin{proof}
We will prove first that $I=I^\sat.$ In this case, $\deg(\VI)=6$ by Lemma~\ref{Lemma4.4}. Moreover, with the notations in the proof of Lemma~\ref{Lemma4.4}, we have
$$H_\m^1(R/I^\sat)\simeq (J/\mathfrak{b})^\ast[-3]$$
and $\regu(R/J)=1$. Hence $\dim_k(R/J)_\mu=\deg(\VJ)=3$ for all $ \mu \geq 1.$ It follows that
$$\dim_kH_\m^1(R/I^\sat)_\mu=0,\; \forall \mu \geq 2, \; \dim_kH_\m^1(R/I^\sat)_1=3.$$
		
Therefore $\regu(R/I^\sat)=2,$ hence $I^\sat$ is generated in degree at most 3. Since $\indeg(I^\sat)=3$, it is immediate that $I^\sat$ is exactly generated in degree 3.  Furthemore,
$$\dim_k(I^\sat)_3=\dim_k R_3 -\dim_k (R/I^\sat)_3=10-6=4,$$
which implies  $I^\sat=I.$ \\
		
By the Hilbert-Burch Theorem, $I$ admits a free resolution of the form
$$\xymatrix {0\ar[r] & R(-4)^3\ar[r]^{M}& R(-3)^4\ar[r]& R\ar[r]& R/I\ar[r]& 0.}$$
By changing the coordinates, there is no loss of generality in assuming the matrix $M$ be of the form
$$M=\pmt{X_1& 0& 0\\ a_1 X_2& X_2& 0\\ b_1 X_3& b_2 X_3& X_3\\ c_1 L& c_2 L& c_3L}$$
where $L=\alpha_1 X_1 + \alpha_2 X_2+\alpha_3 X_3.$ One has $c_3\neq 0$ since $f_i\neq 0$ for all $i=0,1,2,3.$ Since $I$ is locally a complete intersection, $V(I_2(M)) = \emptyset,$ where $I_2(M)$ is the ideal of $(2\times 2)$-minors of $M$. It follows that $\alpha_1\alpha_2\alpha_3\neq 0$.\\
		
Let now a fiber over a closed point $\point$  of coordinates $(\lambda_1:\lambda_2:\lambda_3:\lambda_4)\in \PP_k^3$ be of dimension 1 and its unmixed component be defined by $h_\point\in R$. It follows from \cite[Lemma~10]{Botbol-Buse_Chardin14} that
$$h_\point=\gcd(\ell_1,\ell_2,\ell_3),$$
where
\begin{align*}
\ell_1&=(\lambda_1+c_1\alpha_1\lambda_4)X_1+(a_1\lambda_2+c_1\alpha_2\lambda_4)X_2+(b_1\lambda_3+c_1\alpha_3\lambda_4)X_3,\\
\ell_2&=c_2\alpha_1\lambda_4X_1+(\lambda_2+c_2\alpha_2\lambda_4)X_2+(b_2\lambda_3+c_2\alpha_3\lambda_4)X_3,\\
\ell_3&=c_3\alpha_1\lambda_4X_1+c_3\alpha_2\lambda_4X_2+(\lambda_3+c_3\alpha_3\lambda_4)X_3.
\end{align*}
$\bullet $ If $\lambda_4=0,$ then there are the following cases:
\begin{enumerate}
\item[Case 1:] $\lambda_3\neq 0.$ In this case we have $\lambda_1=\lambda_2=0$. Therefore $\point=(0:0:1:0)$ and $h_\point =X_3.$
\item[Case 2:] $\lambda_3 = 0$ and $\lambda_2\neq 0$. It follows that $\lambda_1=0$. Thus $\point=(0:1:0:0)$ and $h_\point =X_2.$
\item[Case 3:]  $\lambda_3 = 0$ and $\lambda_2= 0$. We must have $\lambda_1\neq 0$. Hence  $\point=(1:0:0:0)$ and $h_\point =X_1.$
\end{enumerate}
$\bullet $ If $\lambda_4=1,$ then from three polynomials $\ell_1, \ell_2,\ell_3$ have a common divisor, we have
\begin{align*}
&\frac{\lambda_1+c_1\alpha_1}{c_3\alpha_1}=\frac{a_1\lambda_2+c_1\alpha_2}{c_3\alpha_2} =\frac{b_1\lambda_3+c_1\alpha_3}{\lambda_3+c_3\alpha_3},\\
&\frac{c_2}{c_3}=\frac{\lambda_2+c_2\alpha_2}{c_3\alpha_2} =\frac{b_2\lambda_3+c_2\alpha_3}{\lambda_3+c_3\alpha_3}.
\end{align*}
These equalities imply that $\lambda_1=\lambda_2=0.$ Furthermore, we also have $\lambda_3=0$. Indeed, if $\lambda_3 \neq 0$, then $c_1=b_1c_3$ and $c_2=b_2c_3$, contrary to $f_i\neq 0$ for all $i=0,\ldots, 3$. Therefore $\point=(0:0:0:1)$ and $h_\point =L.$
\end{proof}
	
\begin{exem}   Consider the parameterization given by
\begin{align*}
f_0= X_2X_3(X_1+X_2+X_3),\qquad &f_2=X_1X_2(X_1+X_2+X_3),\\
f_1=X_1X_3(X_1+X_2+X_3),\qquad &f_3=X_1X_2X_3.
\end{align*}

Using  {\tt Macaulay2} \cite{Macaulay2}, we see that $\VI$ is locally a complete intersection of degree $6$ and that $I$  is a saturated ideal with free resolution
$$\xymatrix {0\ar[r] & R(-4)^3\ar[r]^{M}& R(-3)^4\ar[r]& R\ar[r]& R/I\ar[r]& 0,}$$
where the matrix $M$ is given by
$$\pmt{ 0& 0& X_1\\ X_2& -X_2 & 0 \\-X_3& 0& 0\\ 0 & X_1+X_2+X_3& -(X_1+X_2+X_3) } .$$
Thus we obtain $\Y=\{\point_1,\point_2,\point_3,\point_4\}$	with
\begin{align*}
\point_1& =(1:0:0:0), \quad \quad h_{\point_1}=X_1,\qquad\point_2=(0:1:0:0),\qquad h_{\point_2}=X_2,\\
\point_3 & =(0:0:1:0), \quad \quad h_{\point_3}  =X_3,\qquad \point_4=(0:0:0:1),\;\quad h_{\point_4}=X_1+X_2+X_3.
\end{align*}
Consequently, we have
$$\sum_{\point\in \Y}\deg(h_\point)=4.$$
\end{exem}

Now we state our main result.
\begin{theo} 
Let $I$ be a homogeneous ideal of $R$ such that $\indeg(I^\sat)=\indeg(I)=d\geq 4.$ If $\VI=\Proj(R/I)$ is locally a complete intersection of dimension zero, then
$$\sum_{\point\in \Y}\deg(h_\point)\leq \left \lfloor \frac{d}{2} \right \rfloor  d-1$$
where $\lfloor x \rfloor:=\max \{ n\in \Z\, |\, n\leq x \}.$
\end{theo}
\begin{proof} Set $\nu = \lfloor \frac{d}{2}\rfloor $. By Proposition~\ref{Proposition4.1}, it suffies to show that $(I^\nu)^\sat_{\nu d-1}\neq 0$. Suppose on the contrary that  $(I^\nu)^\sat_{\nu d-1}=0$.  By  Proposition~\ref{Proposition4.5}, one has
$$\dim_k H_\m^2(I^\nu)_{\nu d-1}= \frac{1}{2}\nu(\nu+1)\deg(\VI)-\frac{1}{2}\nu d(\nu d+1).$$
We now show that this is not possible since $\deg(\VI)\leq d^2-2d+3$ by Lemma~\ref{Lemma4.4}. Indeed, consider two cases:
\begin{enumerate}
\item[(i)] If $d=2s$ with $s\geq 2$, then $\nu= \lfloor \frac{d}{2}\rfloor = s$. Therefore 
\begin{align*}
\dim_k H_\m^2(I^s)_{sd-1}& = \frac{1}{2}s\big((s+1)\deg(\VI) -2s(2s^2+1)\big)\\
&\leq \frac{1}{2}s \big((s+1)(4s^2-4s+3) -2s(2s^2+1)\big)\\
&\leq -\frac{3}{2}s(s-1)<0,
\end{align*}
which is impossible.
\item[(ii)]  Similarly, if $d=2s+1$ with $ s\geq 2$, then $\nu= \lfloor \frac{d}{2}\rfloor = s$. Thus $\dim_k H_\m^2(I^s)_{sd-1}\leq -\frac{1}{2}s(s-1)<0,$ 	which is impossible.
\end{enumerate}
\end{proof}
	
\begin{exem} Let $d\geq 4 $ be an integer. Consider the parameterization given by
\begin{align*}
\begin{array}{ccc}
f_0= X_1^{d-3}X_2(X_1^2-X_2^2),&&f_2=X_1^{d-3}X_3(X_2^2-X_3^2),\\
f_1=X_1^{d-3}X_3(X_1^2-X_2^2),&&f_3=X_2^{d-3}X_3(X_2^2-X_3^2).
\end{array}
\end{align*}
We see that $\VI$ is locally a complete intersection of degree $d^2-3d+7$. Moreover, $I$ is a saturated ideal with free resolution
$$\xymatrix {0\ar[r] & R(-d-1)\oplus R(-d-2)\oplus R(-2d+3)\ar[r]^{\qquad \qquad\qquad M}& R(-d)^4\ar[r]& R\ar[r]& R/I\ar[r]& 0,}$$
where the matrix $M$ is given by
$$\pmt{-X_3& 0& 0\\ X_2& X_2^2-X_3^2& 0\\ 0& -X_1^2 +X_2^2& -X_2^{d-3}\\ 0& 0& X_1^{d-3}} .$$
We consider the following cases.\\
		
\underline{Case 1:}  if $d-3$ is odd, we obtain  $\Y=\{\point_1,\point_2,\point_3,\point_4,\point_5,\point_6\}$
with
\begin{align*}
\point_1& =(1:0:0:0), \quad \quad h_{\point_1}=X_3,\qquad\quad\;\point_2=(0:0:0:1),\qquad h_{\point_2}=X_1^{d-3},\\
\point_3 & =(1:1:0:0), \quad \quad h_{\point_3}  =X_2-X_3,\quad \point_4=(1:-1:0:0),\quad h_{\point_4}=X_2+X_3,\\
\point_5 & =(0:0:1:1),  \quad\quad h_{\point_5}=X_1-X_2, \quad 	\point_6  =(0:0:1:-1),  \quad h_{\point_6}=X_1+X_2.
\end{align*}
		
\underline{Case 2:} if $d-3$ is even, we have $\Y=\{\point_1,\point_2,\point_3,\point_4,\point_5\}$
with
\begin{align*}
\point_1& =(1:0:0:0), \quad \quad h_{\point_1}=X_3,\qquad\quad\;\point_2=(0:0:0:1),\qquad h_{\point_2}=X_1^{d-3},\\
\point_3 & =(1:1:0:0), \quad \quad h_{\point_3}  =X_2-X_3,\quad \point_4=(1:-1:0:0),\quad h_{\point_4}=X_2+X_3,\\
\point_5 & =(0:0:1:1), \quad\quad h_{\point_5}=X_1^2-X_2^2.
\end{align*}\\
		
Therefore
$$\sum_{\point\in \Y}\deg(h_\point)=d+2.$$
\end{exem}
\begin{ques}
Is $\sum_{\point\in \Y}\deg(h_\point)$ bounded  linearly in the degree $d$?
\end{ques}

\section*{Acknowledgments}
This work was done as a part of the author’s Ph.D. thesis at Pierre and Marie Curie University (UPMC), France. He would like to thank his advisor, Professor Marc Chardin, for suggesting the problem, for his enormous help, for being so supportive and his careful reading of the manuscript.



\end{document}